\documentclass[11pt]{amsart}

\usepackage{amsmath}
\usepackage{amssymb}
\usepackage{amsfonts}
\usepackage{amsthm}
\usepackage{enumerate}
\usepackage{bbm}

\usepackage{graphicx}
\usepackage{verbatim}

\newcommand{\wh}{\widehat}

\newcommand{\indicator}[1]{\mathbbm{1}_{ {#1}  }}

\newcommand{\F}{\mathcal F}

\newcommand{\ci}[1]{_{{}_{\scriptstyle{#1}}}}

\newcommand{\Be}{\begin{equation}}
\newcommand{\Ee}{\end{equation}}

\newcommand{\Bm}{\begin{multline}}
\newcommand{\Em}{\end{multline}}
\newcommand{\Bea}{\begin{eqnarray}}
\newcommand{\Eea}{\end{eqnarray}}
\newcommand{\Beas}{\begin{eqnarray*}}
\newcommand{\Eeas}{\end{eqnarray*}}
\newcommand{\Benu}{\begin{enumerate}}
\newcommand{\Eenu}{\end{enumerate}}
\newcommand{\Bi}{\begin{itemize}}
\newcommand{\Ei}{\end{itemize}}

\def\fz{{z}}
\def\dz{{\frak z}}
\def\gga{\Gamma}
\def\ccV{{V}}

\def\intslash{\rlap{\kern  .32em $\mspace {.5mu}\backslash$ }\int}
\def\qsl{{\rlap{\kern  .32em $\mspace {.5mu}\backslash$ }\int_{Q_x}}}
\def\Re{\operatorname{Re\,}}

\def\vth{\vartheta}

\def\Q{\mathcal Q}

\def\emph#1{{\it #1 }}

\def\ga{\gamma}

\def\cf{{\it cf}}

\def\rad{{\text{\it rad}}}

\def\inn#1#2{\langle#1,#2\rangle}

\def\noi{\noindent}

\def\meas{{\text{\rm meas}}}

\def\lc{\lesssim}

\def\ka{\kappa}
            
\def\la{\lambda}

\def\om{\omega}              \def\Om{\Omega}

\def\fC{{\mathfrak {C}}}

\def\fS{{\mathfrak {S}}}

\def\fV{{\mathfrak {V}}}

\def\fm{{\mathfrak {m}}}

\def\fv{{\mathfrak {v}}}

\def\bbR{{\mathbb {R}}}

\def\bbZ{{\mathbb {Z}}}

\def\cE{{\mathcal {E}}}
\def\cF{{\mathcal {F}}}

\def\cK{{\mathcal {K}}}

\def\cQ{{\mathcal {Q}}}

\def\cS{{\mathcal {S}}}
\def\cT{{\mathcal {T}}}

\def\cW{{\mathcal {W}}}
\def\cX{{\mathcal {X}}}

\def\Q{{\hbox{\bf Q}}}

 at 10 true pt

\def\be#1{\begin{equation}\label{ #1}}
\def\endeq{\end{equation}}
\def\endal{\end{align}}
\def\bas{\begin{align*}}
\def\eas{\end{align*}}
\def\bi{\begin{itemize}}
\def\ei{\end{itemize}}

\def\emph#1{{\it #1}}
\def\textbf#1{{\bf #1}}

\theoremstyle{plain}
   \newtheorem{theorem}{Theorem}[section]

   \newtheorem{proposition}[theorem]{Proposition}
   
   \newtheorem{lemma}[theorem]{Lemma}
   \newtheorem{corollary}[theorem]{Corollary}

   \newtheorem{theorem*}{Theorem}

\theoremstyle{remark}

\theoremstyle{definition}

\numberwithin{equation}{section}

\begin{document}

\title{On  radial and conical Fourier  multipliers}

\author[Y. Heo \ \ \ F. Nazarov \ \ \ A. Seeger]{Yaryong Heo \ \ \ F\"edor Nazarov  \ \ \  Andreas Seeger}

\address{
Department of Mathematics\\ University of Wisconsin-Madison\\Madison, WI 53706, USA}
\email{heo@math.wisc.edu}
\email{nazarov@math.wisc.edu}
\email{seeger@math.wisc.edu}

\subjclass{42B15}

\begin{thanks} {Y.H. supported by  Korea Research Foundation Grant KRF-2008-357-C00002
and National Research Foundation of Korea Grant NRF-2009-0094068.
F.N. supported in part by NSF grant 0800243.
A.S. supported in part by NSF grant 0652890.}
\end{thanks}


\begin{abstract}
We investigate connections between radial Fourier multipliers  on $\bbR^d$
and  certain  conical Fourier multipliers on $\bbR^{d+1}$. As an
application we obtain a new  weak type endpoint bound for the Bochner-Riesz
multipliers associated to the light cone in $\bbR^{d+1}$, where $d\ge 4$,
and results on characterizations of $L^p\to L^{p,\nu}$ inequalities for
convolutions with radial kernels.
\end{abstract}
\maketitle
\section*{Introduction}
This paper is a sequel to \cite{hns} in which the authors obtained a
characterization of radial multipliers of $\cF L^p(\bbR^d)$ provided that
$1<p<2 $ and the dimension $d$ is large enough.
The main estimate in \cite{hns} was concerned with a convolution
inequality for surface measure on spheres  which in this paper we
state as Hypothesis $\text{Sph}(p_1,d)$  for some $p_1>1$. 
Under this hypothesis we shall  prove  several equivalent statements on cone
multipliers and radial Fourier multipliers.


In what follows we fix   
a radial $C^\infty(\bbR^d)$
function $\psi_\circ$  supported in  a small  ball of radius
centered at the origin  (say, of radius $\le (100 d)^{-1}$) 
whose Fourier transform vanishes at the origin to high order (say 
$100d$). We assume that $\wh \psi_\circ(\xi)\neq 0$ for $1/8\le |\xi|\le
8$. Set
\Be \label{psidef} \psi=\psi_\circ*\psi_\circ
\Ee
and,
for $y\in \bbR^d$ and for $r\ge 1$, let $\sigma_r$ be surface measure
on the sphere of radius $r$ centered at the origin, i.e.
\Be \inn{\sigma_r}{f}= r^{d-1} \int_{S^{d-1}} f(ry') d\sigma_1(y')\,.
\label{sigmar}\Ee
\medskip

\noi{\bf Hypothesis $\text{Sph}\mathbf{(p,d)}$.} {\it There is a constant $C$
  so that for every
$h\in L^{p}(\bbR^d\times \bbR^+;dy\, r^{d-1}dr)$
the inequality
\begin{multline}\label{mainlpcont}
\Big\| \int_{\bbR^d}\int_1^\infty
h(y,r) \sigma_r*\psi(\cdot -y) \, dr dy \Big\|_{L^{p}(\bbR^d)} \\ \le C
\Big(\iint_{\bbR^d\times \bbR^+} |h(y,r)|^p dy\, r^{d-1}dr\Big)^{1/p}
\end{multline}
holds.}

\begin{theorem} \label{hnsthm}(\cite{hns}) {\it  Hypothesis Sph(p,d) holds for $d\ge 4$,
$1\le p<\frac{2(d-1)}{d+1}$.}
\end{theorem}

\section{Statement of results}\label{statements}
In what follows
$L^{p,\nu}$ denotes the
standard Lorentz space, and we shall usually assume that $p\le \nu\le\infty$.
We denote by  $\cF_d f$  the $\bbR^d$ Fourier transform  of $f$,
defined  by $\cF_d f(\xi)=\int f(y) e^{-i\inn {y}{\xi}}dy$.
We shall also write $\cF f $ or $\widehat f$ if the dimension is clear
from the context.

For each $k\in \bbZ$ let $\gamma_k $ be 
supported in
 $(-1/4,1/4)$. Define
\Be  m(\xi,\tau)\equiv  m_\gamma(\xi,\tau)
= \sum_{k\in \bbZ} \gamma_k\big(\frac{|\xi|-\tau}{2^k}\big) \,
 \indicator{[2^k,
  2^{k+1})}(\tau)
\Ee
where $\indicator{E}$ denotes the characteristic function of $E$.
Let
$T$ denote the operator acting  on Schwartz functions in $\bbR^{d+1}$ by
\Be \label{Tdplus1}
\cF_{d+1}[Tf] (\xi,\tau) = m_\gamma(\xi,\tau) \cF_{d+1}f(\xi,\tau).
\Ee
Moreover,
for each fixed $\tau\in (0,\infty) $, define
an operator
$T^{\tau}$ on functions in $\bbR^{d}$ by
\Be\label{Ttau}\cF_d[T^{\tau}\!  f] (\xi) =
 \gamma_k\big(\frac{|\xi|-\tau}{2^k}
\big)
\F_d f (\xi), \quad\text  { if } \tau\in [2^k, 2^{k+1})
.\Ee

\begin{theorem} \label{dLconverse}
Let $T$, $T^\tau$ be as in \eqref{Tdplus1}, \eqref{Ttau}.

Suppose that
$1<p_1<\frac{2d}{d+1}$
and suppose that Hypothesis $\text{Sph}(p_1,d)$ holds.
Let $1<p<p_1$, $p\le \nu\le\infty$.
Then the following statements are equivalent.

(i) $T$ maps $L^p(\Bbb R^{d+1}) $ boundedly to
$L^{p,\nu}(\bbR^{d+1})$.

(ii) There is a constant $C_p$ so that for all sequences $\{\tau_k\}_{k=-\infty}^\infty$ satisfying
$\tau_k\in [2^k,2^{k+1})$, and for all $f\in L^p(\bbR^d)$
\[
\Big\|\sum_{k\in \bbZ} \alpha_k T^{\tau_k}\! f\Big\|_{L^{p,\nu}(\bbR^d)}
\le C_p \sup_{k\in \bbZ}|\alpha_k| \,\|f\|_{L^p(\bbR^d)}
\]

%
(iii) For every $k\in \bbZ$ there  is a $\tau_k\in [2^k, 2^{k+1})$ such that
 $T^{\tau_k} $ maps $L^p(\Bbb R^{d}) $ boundedly to
$L^{p,\nu}(\bbR^{d})$, and such that
$\sup_k \|T^{\tau_k}\|_{L^p\to L^{p,\nu}}<\infty$.

(iv) The functions $s\mapsto \widehat \gamma_k(s) \,(1+|s|)^{-\frac{d-1}{2}} $
 belong to the weighted Lorentz space
$L^{p,\nu}(\bbR, (1+|\cdot|)^{d-1})$, with the uniform bound
\Be\label{gammakhat}
 \sup_{k\in \bbZ}\, \biggl\|\frac{\widehat \gamma_k}
 { (1+|\cdot|)^{\frac{d-1}2}}
\biggl\|_{L^{p,\nu}(\bbR,
   (1+|r|)^{d-1}dr)}
<\infty\, .
\Ee

(v) The functions $\cF_d^{-1} [\gamma_k(|\cdot|)]$
belong to
$L^{p,\nu}(\bbR^d)$, with the uniform bound
$$\sup_{k\in \bbZ}
\big\|\cF_d^{-1} [\gamma_k(|\cdot|)]\big\|_{L^{p,\nu}(\bbR^d)}\,<\, \infty\,.
$$
\end{theorem}

From Theorem \ref{hnsthm} we get
\begin{corollary} \label{dLconversecor}
Statements (i)-(v) in Theorem \ref{dLconverse}
 are equivalent if $d\ge 4$, $1<p<\frac{2(d-1)}{d+1}$, $p\le\nu\le\infty$.
\end{corollary}

The equivalence of (iv)$\iff$(v) and 
the implication
(iii)$\implies$(iv) are  in \cite{gs}.
 The implication (ii)$\implies$(iii) is trivial.
The
implication $\text{(i)}\implies\text{(iii) }$
follows from a version 
of de Leeuw's theorem, see Lemma \ref{leeuwlemma}.
It is not presently
clear how to deduce the global statement (ii) directly from (i), without going through (iv) or (v).
The proofs of the main implications $\text{(iv)}\implies\text{(i)}$ and
$\text{(iv)}\implies\text{(ii)}$  are
 given in \S\ref{mainestimate}, \S\ref{pfdLconverse};
they rely on  Hypothesis $\text{Sph}(p_1,d)$.

As a consequence of the  implication (iv)$\implies$(i)   we shall
derive a new  endpoint result
for the so-called Bochner-Riesz multipliers for the cone, defined by
\Be\label{br}\rho_\la(\xi,\tau)= \Big(1-\frac{|\xi|^2}{\tau^2}\Big)^\la_+.
\Ee
It is conjectured that $\rho_\la$ is a multiplier of $\cF L^p(\bbR^{d+1})$
if  $\la >d(1/p-1/2)-1/2$ and $1<p<\frac{2d}{d+1}$; this condition is necessary. This
conjecture is open in the full $p$-range. The first sharp $L^p$
results for some range of $p$ were proved by T. Wolff \cite{wolff}
in two dimensions, with extensions and improvements in
\cite{LW}, \cite{gs-lsm}, \cite{gss}, \cite{heo}, \cite{hns}.
For the endpoint $\la=d(1/p-1/2)-1/2$ one conjectures a weak type
$(p,p)$ inequality for $p<\frac{2d}{d+1}$. 
This endpoint inequality  cannot be replaced by a stronger statement
such as $L^p\to L^{p,\nu}$ boundedness for $\nu<\infty$. In
\S\ref{conemultwt}
we prove

\begin{corollary}
\label{wtbr}
Let $d\ge 2$ and $p_1>1$, and suppose that Hypothesis $\text{Sph}(p_1,d)$
holds. Let $\rho_\la$ be as
in \eqref{br}. If $\la =d(1/p-1/2)-1/2$ and $1<p<p_1$  then
\Be\label{coneLpinftyestimate}
\big\|\cF^{-1} [\rho_\la \widehat f]\big\|_{L^{p,\infty}(\bbR^{d+1})}
\le C_p \|f\|_{L^p(\bbR^{d+1})}
\Ee
for all $f\in L^p(\bbR^{d+1})$.
In particular,  \eqref{coneLpinftyestimate} holds for
$d\ge 4$ and $1<p<\frac{2(d-1)}{d+1}$.
\end{corollary}

{\it Remark.}
Sharp  bounds in
 $H^p$, $p<1$ 
and sharp bounds for the
operator acting on functions of the form $f(x,t)= f_0(|x|,t)$ can be
found in Hong's articles \cite{hong1},  \cite{hong2}. More recently, Heo, Hong and  Yang  \cite{hhy} proved a
weak type $(1,1)$ inequality for a localized cone multiplier
$\chi(\tau) \rho_{(d-1)/2}(\xi,\tau)$, in dimension $d\ge 4$.
 As a  corresponding result for the
global cone multiplier  one can prove that for 
$\Re(\la)=(d-1)/2$ the 
operator
$f\to \cF_{d+1}^{-1} [\rho_\la \widehat f]$ is bounded
 from the Hardy space $H^1$ to $L^{1,\infty}$, under the assumption
 that $Sph(p_1,d)$ holds for some $p_1>1$. 
 This can be obtained by an analytic interpolation argument using  
 the analytic family of multipliers $\la\to \rho_\la$, the $H^p\to
 L^{p,\infty}$ bounds in \cite{hong1}  for $p<1$ and
 $\Re(\la)=d(1/p-1/2)-1/2$, and the $L^p$ result
 of Corollary \ref{wtbr}. For the justification of the analytic
 interpolation one uses an adaptation of arguments in
 \cite{sagher}. We shall not give details here.

\medskip

The equivalence of condition  (ii) in the theorem with conditions
(iv) or (v) immediately 
yields a  generalization of the main result in \cite{hns} to 
$L^p\to L^{p,\nu}$ inequalities.

\begin{corollary}
\label{mainthmLor}
Let $p_1>1$, $1<p<p_1<\frac{2d}{d+1}$  and assume that Hypothesis
$\text{Sph}(p_1,d)$
 holds.
 Let  $m=m_0(|\cdot|)$ be  a bounded radial function on $\bbR^d$ and
define $\cT_m$ by
$$\cF_d[{\cT_m f}] =m\cF_d f.$$
Then, for  any Schwartz function $\eta\neq 0$
\begin{equation}\label{equivlor}
\big\|\cT_m\big\|\ci{L^p\to L^{p,\nu}}\,\approx
\,\sup_{t>0}\,t^{d/p}\big\|\cT_m[\eta(t\cdot)]\big\|\ci{L^{p,\nu}}\,.
\end{equation}
Moreover, if $\phi\in C^\infty_c$ is compactly supported in
$(0,\infty)$ (and not identically zero) and 
$\ka_t$ denotes the Fourier transform on $\bbR$ of the function $\phi
m_0(t\cdot)$ then
\begin{equation}\label{m0}\big\|\cT_m\big\|\ci{L^p\to L^{p,\nu}}\,\approx
\sup_{t>0}\Big\| \frac{\ka_t}{ (1+|\cdot|)^{\frac{d-1}2}}
\Big\|_{L^{p,\nu} (\bbR;(1+|r|)^{d-1}dr)}<\infty.
\end{equation}
\end{corollary}

The equivalence of the two conditions on the right hand sides of
\eqref{equivlor} and  \eqref{m0} with $L^p_\rad \to L^{p,\nu}$
boundedness
(i.e. on radial functions, for $p<\frac{2d}{d+1}$) was proved in \cite{gs}.
The $L^p$  case  ($p=\nu$) for $1<p<\frac{2(d-1)}{d+1}$   was proved in \cite{hns}, moreover that
article has already $L^p\to L^{p,\nu}$ inequalities  for radial
multipliers which  are compactly supported away from the origin.

\section{Preliminaries}

The following dyadic interpolation lemma is convenient in dealing with
the short range estimation in \S\ref{mainestimate}; it is proved in \S2 of
\cite{hns}.

\begin{lemma} \label{dyadicinterpol}
Let $0<p_0<p_1<\infty$.
Let $\{F_j\}_{j\in\bbZ}$ be a sequence of measurable functions on a measure space
$\{\Omega, \mu\}$,
and let $\{s_j\}$ be a sequence of nonnegative numbers.
Assume that, for all $j$,
the inequality
\begin{equation*}
\|F_j\|_{p_\nu}^{p_\nu}\le 2^{j p_\nu} M^{p_\nu} s_j
\end{equation*}
holds  for $\nu=0$ and $\nu=1$.
Then for every $p\in (p_0,p_1)$, there is a constant $C=C(p_0,p_1,p)$ such
 that
\begin{equation*}
\Big\| \sum_j F_j\Big\|_p^p \le C^p M^p \,\sum_j 2^{jp} s_j\,.
\end{equation*}
\end{lemma}

We need a simple fact  about Lorentz spaces.

\begin{lemma}\label{fubinilemma}
Let $(\cX_1,\mu_1)$, $(\cX_2,\mu_2)$ be $\sigma$-finite
measure spaces,
and let
$\mu=\mu_1\times\mu_2$ be
 the product measure on $\cX_1\times\cX_2$. Then, for $1\le
p<\infty$,  $p\le \nu\le \infty$, and any $\mu$-measurable function  $G$,
\begin{equation}\label{fubinilor}
\|G\|_{L^{p,\nu}(\cX_1\times\cX_2,\mu)}\le C_{p,\nu}
\Big(\int \|G(x_1,\cdot)\|_{L^{p,\nu}(\cX_2,\mu_2)}^p d\mu_1\Big)^{1/p}.
\end{equation}
\end{lemma}
The proof is a Fubini-type argument (in conjunction with Minkowski's inequality in $\ell^{\nu/p}$), we  refer to \S9 of \cite{hns}.

Finally we need a version of a restriction theorem for
multipliers due to de Leeuw.

\begin{lemma} \label{leeuwlemma}
Let $1<p<\infty$ and $1\le p\le \infty$ and let
$m$ be a bounded continuous function in $\bbR^{d+1}$. Suppose that the
  operator $f\mapsto \cF_{d+1}^{-1}[m\cF_{d+1} f]$ is bounded from
$L^{p,\nu_1}$ to $L^{p,\nu_2}$ with operator norm $B$. Let, for $\xi\in\bbR^d$,  $m_0(\xi)= m(\xi,0)$. Then there is a constant
  $C$ independent of $m$ and $f$ such that
$$\big\|\cF^{-1}_d[m_0 \cF_d f]\big\|_{L^{p,\nu_2} }\le
  CB\big\|f\big\|_{L^{p,\nu_1}}.$$
\end{lemma}
\begin{proof} This is just a modification of the proof given in
  \cite{jodeit}. By the hypothesis
\Be\label{multindplusone}
\Big|\iint m(\xi,\tau) \widehat F(\xi,\tau) \widehat G(\xi,\tau)
d\xi d\tau
\Big| \le B \|F\|\ci{L^{p,\nu_1}} \|G\|_{L^{p',\nu_2'}}\,.
\Ee
Now let $\chi$ be a Schwartz function on $\bbR$ whose Fourier transform
is supported in $(-1,1)$.
 so that $\widehat \chi(0)\neq 0$. Given a small
parameter $\delta$ we let $\chi_\delta(t)=\chi(\delta t)$, and, for
 $f\in L^{p,\nu_1}(\bbR^d)$, $g\in
  L^{p',\nu_2'}(\bbR^d)$ we define $F_\delta(x,t) =\delta\chi_\delta
( t)f(x)$
and $G_\delta(x,t) =\chi_\delta (t)g(x)$.
Observe that the inequality
$$\|h\otimes \chi_\delta\|_{L^{q,r}(\bbR^{d+1})} \le C(\chi)
\delta^{-1/q}
\|h\|_{L^{q,r}(\bbR^{d+1})}$$
is immediate for $q=r$, by Fubini, and then holds
 for arbitrary $r$ by real interpolation.
Thus $\|F_\delta\|_{L^{p,\nu_1}} \le \delta^{1-1/p} \|f\|_{L^{p,\nu_1}}$ and
$\|G_\delta\|_{L^{p',\nu_2'}} \le \delta^{-1/p'} \|g\|_{L^{p',\nu_2}}$.
 Apply \eqref{multindplusone}
  with $F_\delta$, $G_\delta$ and  let  $\delta\to 0$.
This yields
\begin{align*}&[\widehat \chi(0)]^2
\Big|\int m(\xi,0) \widehat f(\xi) \widehat g(\xi)
d\xi \Big|\\&=\lim_{\delta\to 0}
\Big|\iint m(\xi,\tau) \widehat f(\xi) \widehat
g(\xi)\delta^{-1}[\widehat \chi(\delta^{-1}\tau)]^2 d\xi d\tau \Big|\\
 &\le C B \|f\|_{L^{p,\nu_1}} \|g\|_{L^{p',\nu_2'}}
\end{align*}
which implies the assertion.
\end{proof}

\section{Inequalities for spherical measures}\label{prev}
We shall now derive a consequence
of Hypothesis $\text{Sph}(p_1,d)$ which will be used in
conjunctions with atomic decompositions. Similar inequalities have been used in \cite{hns} but they were derived using the proof of
\eqref{mainlpcont}
rather than \eqref{mainlpcont}  itself.

In what follows let $\ell$ be a nonnegative integer and, for 
 $\dz=(z_1,\dots, z_d)\in \bbZ^d$, let
\Be\label{Rzl}R_{\dz,\ell}= \{x\in \bbR^d: 2^{\ell}z_i\le
x_i<2^{\ell+1}z_i,
 i=1,\dots, d\}\Ee
so that the $R_{\dz,\ell}$ form a tiling of $\bbR^d$ with dyadic cubes of
sidelength $2^{\ell}$.  We denote by $\chi\ci{R_{\dz,\ell}}$ the characteristic function of $R_{\dz,\ell}$.
We denote variables in
$\bbZ^{d+1}$ by $z=(\dz,z_{d+1})$ with $\dz\in \bbZ^d$.
Let
$I_{\fz_{d+1},\ell}
= [2^{\ell}\fz_{d+1},2^{\ell+1}\fz_{d+1})$ and let \Be\label{chizl}\chi_{\fz,\ell} (x,t):=\chi\ci{R_{\dz,\ell}}(x) 
\chi\ci{I_{\fz_{d+1}},\ell}(t).\Ee
For each $r>0$, $\fz\in\bbZ^{d+1}$ we are given an $L^2(\bbR^{d+1})$ function
$g_{\fz,r}$
depending continuously on $r$ such  that
\Be\label{Gzr}\|g_{\fz,r}\|_{L^2(\bbR^{d+1})}\le 1, \quad \text{ for all $\fz,r$}. \Ee
Moreover, for each positive integer $n$ we are given an $L^1$ function
$\om_n$ supported on $[1/2,2]$ so that
\Be\label{omeganL1}\sup_n\int_{1/2}^2 |\om_n(\rho)| d\rho \le 1.
\Ee
 Let $\ell>0$.
We
define an operator $\cS_\ell$ acting on functions $F$
on $\bbZ^{d+1}\times \bbR^+$  by
\begin{multline}\label{Slext}
\cS_\ell F(x,t) =\\ \sum_{\fz}\sum_{n=\ell}^\infty \int_{2^n}^{2^{n+1}} F(\fz,r) \,
\int_{1/2}^2\om_n(\rho) \int_{\bbR^{d}}\psi*\sigma_{\rho r}(x-y)
 [\chi_{\fz,\ell}\,g_{\fz,r}](y, t-r) \,dy\, d\rho\,dr.
\end{multline}
On the set $\bbZ^{d+1}\times \bbR^+$ we define the measure
$\mu_d$ by
$$\mu_d(E)= \int_0^\infty  \sum_{\fz\in\bbZ^{d+1}:(\fz,r)\in E }
r^{d-1}dr$$
for a measurable set $E\subset \bbZ^{d+1}\times \bbR^+$.

\begin{proposition} \label{propext}
Suppose $d\ge 2$ and Hypothesis $\text{Sph}(p_1,d)$
holds for some $p_1\in (1,2)$.
Let
$g_{\fz,r}$, $\omega_n$ be  as in \eqref{Gzr}, \eqref{omeganL1}, $\ell>0$,
 and define $\cS_\ell$ by \eqref{Slext}.
Then
the inequality
\Be\label{Slineqext}
\big\|\cS_\ell F\big\|_{L^{p,\nu}(\bbR^{d+1})} \le C_{p,\nu} \,2^{\ell (d+1)(\frac 1p-\frac 12)-\alpha}
\|F\|_{L^{p,\nu}(\bbZ^{d+1}\times\bbR^+\!,\,\mu_d)}
\Ee
holds for
(i)  for  $p=1=\nu$, with $\alpha=\frac{d-1}{2}$,  (ii) for  $p=p_1=\nu$, with
$\alpha=0$ and,  (iii),  for
$$
1<p<p_1, \quad 0<\nu\le \infty \quad\text{with } \alpha
\le  \frac{d-1}2\,\frac{\tfrac {1}{p}-\tfrac {1}{p_1}}{1-\tfrac{1}{p_1}}
.$$
\end{proposition}

\begin{proof}
Statement (iii)  follows by real interpolation from the cases  $p=\nu=p_1$
and $p=\nu=1$.

We consider the case $p=p_1=\nu$.
In order to
apply
Hypothesis $\text{Sph}(p_1,d)$ we interchange the $\rho$-  and the $r$-integrals and change variables $s=r\rho$. This yields
\begin{multline*}
\cS_\ell F(x,t) =
\sum_{\fz}
\sum_{n=\ell}^\infty \int_{\rho=1/2}^2
 \int_{s=2^n\rho}^{2^{n+1}\rho} F(\fz,\tfrac s\rho)\om_n(\rho) \,\times\\
\int_{\bbR^{d}}\psi*\sigma_{s}(x-y)
 [\chi_{\fz,\ell}\,g_{\fz,\frac s\rho}](y, t-\tfrac s\rho) \,dy\, ds
 \frac{d\rho}{\rho}\, ,
\end{multline*}
and thus
\Be \cS_\ell F(x,t) = \int_{2^{\ell-1}}^\infty
\int_{\bbR^{d}}\psi*\sigma_{s}(x-y)
\ccV_\ell F(y,s,t)
\,dy\, ds
\Ee
where
\begin{multline}\label{Velldef} \ccV_{\ell} F(y,s,t):=\\
\int_{\rho=1/2}^2
\sum_{n=\ell}^\infty \om_n(\rho) \chi\ci{[2^n\rho, 2^{n+1}\rho]}(s)
\sum_{\fz}
 F(\fz,\tfrac s\rho) \,
 [\chi_{\fz,\ell}\,g_{\fz,\frac s\rho}](y, t-\tfrac s\rho) \rho^{-1}d\rho\,.
\end{multline}

For fixed $t$ we apply
Hypothesis $\text{Sph}(p_1,d)$ and then integrate in $t$. This yields
\begin{align*}
&\|\cS_\ell F\|_{L^{p_1}(\bbR^{d+1})} =
\Big(\int_t \|\cS_\ell F(\cdot,t)\|_{L^{p_1}(\bbR^d)}^{p_1}dt\Big)^{1/p_1}
\\
&\lc \Big(\int \int_{2^{\ell-1}}^\infty \int \big|  \ccV_\ell F(y,s,t)|^{p_1}
dy \,s^{d-1}ds\, dt\Big)^{1/{p_1}}\,.
\end{align*}
We observe that if $2^\nu<s<2^{\nu+1}$ then only the terms with
$\nu-1\le n\le \nu+1$ contribute to the $n$-sum in \eqref{Velldef}.
Thus, for fixed $(y,t)$,
\begin{multline}\label{intsVellest}
\Big(\int_{2^{\ell-1}}^\infty| \ccV_\ell F(y,s,t)|^{p_1} s^{d-1} ds
\Big)^{1/p_1}
\,\le\, \sum_{i=-1,0,1}
\\
\Big(\sum_{\nu=\ell-1}^\infty \int_{2^{\nu}}^{2^{\nu+1}}
\Big|
\int_{\rho=1/2}^2
 \om_{\nu+i}(\rho)
\sum_{\fz}
 F(\fz,\tfrac s\rho) \,
 [\chi_{\fz,\ell}\,g_{\fz,\frac s \rho}](y, t-\tfrac s\rho) \frac{d\rho}{\rho}
\Big|^{p_1} s^{d-1} ds
\Big)^{1/p_1}\,.
\end{multline}
Now we have for  fixed $\nu$
\begin{align*}
&\Big(\int_{2^{\nu}}^{2^{\nu+1}}
\Big|
\int_{\rho=1/2}^2
 \om_{\nu+i}(\rho)
\sum_{\fz}
 F(\fz,\tfrac s\rho) \,
 [\chi_{\fz,\ell}\,g_{\fz,\frac s\rho}](y, t-\tfrac s\rho) \rho^{-1}d\rho
\Big|^{p_1} s^{d-1} ds
\Big)^{1/p_1}
\\
&\le \int_{1/2}^2 |\om_{\nu+i}(\rho)|
\Big(\int_{2^{\nu}}^{2^{\nu+1}}
\Big|
\sum_{\fz}
 F(\fz,\tfrac s\rho) \,
 [\chi_{\fz,\ell}\,g_{\fz,\frac s\rho}](y, t-\tfrac s\rho)
\Big|^{p_1} s^{d-1} ds
\Big)^{1/p_1}\frac{d\rho}{\rho}
\\
&\le \int_{1/2}^2 |\om_{\nu+i}(\rho)| \rho^{\frac d{p_1}-1}d\rho
\Big(\int_{2^{\nu-1}}^{2^{\nu+2}}
\Big|
\sum_{\fz}
 F(\fz,r) \,
 [\chi_{\fz,\ell}\,g_{\fz,r}](y, t-r)
\Big|^{p_1} r^{d-1} dr
\Big)^{1/p_1}
\\&\lc
\Big(\int_{2^{\nu-1}}^{2^{\nu+2}}
\Big|
\sum_{\fz}
 F(\fz,r) \,
 [\chi_{\fz,\ell}\,g_{\fz,r}](y, t-r)
\Big|^{p_1} r^{d-1} dr
\Big)^{1/p_1}\,.
\end{align*}
We insert this back into
\eqref{intsVellest}
and obtain
\begin{align*}
\Big(\int_{2^{\ell-1}}^\infty| \ccV_\ell &F(y,s,t)|^{p_1} 
s^{d-1} ds\Big)^{1/p_1}
\\
&\lc
\Big(\int_{2^{\ell-1}}^{\infty}
\Big|
\sum_{\fz}
 F(\fz,r) \,
 [\chi_{\fz,\ell}\,g_{\fz,r}](y, t-r)
\Big|^{p_1} r^{d-1} dr
\Big)^{1/p_1}
\\
&\le
\Big(\int_{2^{\ell-1}}^{\infty}
\sum_{\fz} \Big|
 F(\fz,r) \,
 [\chi_{\fz,\ell}\,g_{\fz,r}](y, t-r)
\Big|^{p_1} r^{d-1} dr
\Big)^{1/p_1}\,.
\end{align*}
We take $L^{p_1}$  norms in $(y,t)$ and
perform a shear transformation for fixed $r$ to get
\begin{multline*}
\Big(\int\int \int_{2^{\ell-1}}^\infty| \ccV_\ell F(y,s,t)|^{p_1} s^{d-1} ds\, dt\, dy\Big)^{1/p_1}
\\
\lc
\Big(\int_{2^{\ell-1}}^{\infty}
\sum_{\fz}
 |F(\fz,r)|^{p_1} \iint \,\Big|
 \chi_{\fz,\ell}\,g_{\fz,r}(y, t)
\Big|^{p_1} dt\, dy \,r^{d-1} dr
\Big)^{1/p_1}\,.
\end{multline*}
By H\"older's inequality and our normalizing assumption \eqref{Gzr}
 this is estimated by
\begin{align*}
\Big(\int_{2^{\ell-1}}^{\infty}
\sum_{\fz}
 |F(\fz,r)|^{p_1} 2^{\ell(d+1)(1-p_1/2)}\|\chi_{\fz,\ell}\,g_{\fz,r}\|_2^{p_1}
 r^{d-1} dr
\Big)^{1/p_1}&
\\
\lc  2^{\ell(d+1)(1/p_1-1/2)}
\Big(\int
\sum_{\fz}
 |F(\fz,r)|^{p_1}
 r^{d-1} dr
\Big)^{1/p_1}.&
\end{align*}
This yields the assertion for $p=p_1=\nu$, with $\alpha=0$.


For $p=1$ we 
estimate
\begin{multline*}\big\|\cS_\ell  F\big\|_{L^1(\bbR^{d+1})}
\,\lc \,\sum_{n=\ell}^\infty
\sum_{\fz} \int_{2^n}^{2^{n+1}} |F(\fz,r)|
\int_t \int_{\rho=1/2}^2|\omega_n(\rho)| \,\times
\\
\big\|\psi*\sigma_{r\rho}*
[\chi\ci{R_{\dz,\ell} }
\,g_{\fz,r}] (\cdot, t-r) \big\|_{L^1(\bbR^d)}\,d\rho\,
\chi\ci{I_{\fz_{d+1},\ell}}(t-r) \,dt \,dr.
\end{multline*}
The function
$\psi*\sigma_{r\rho}*[\chi\ci{R_{\dz,\ell} }
g_{\fz,r}] (\cdot, t-r) $ is supported on a set of measure $\le C 2^\ell
r^{d-1}$, namely an annulus of width $\lc 2^{\ell}$ built on a sphere of radius $r\rho$. Moreover
we have
$\big\|\cF_d[\psi*\sigma_{r\rho}]\big\|\le C r^{\frac{d-1}{2}}$ where $C$ is independent of $\rho\in [1/2,2]$.
Thus  the last displayed expression can be estimated by
\begin{align*}
&\sum_{\fz} \sum_{n=\ell}^\infty\int_{2^n}^{2^{n+1}} |F(\fz,r)|
\int_t \int_{\rho=1/2}^2|\omega_n(\rho)| \,2^{\ell/2} r^{(d-1)/2}
\,\times\\
&\qquad\qquad
\big\|\psi*\sigma_{r\rho}*[\chi\ci{R_{\dz,\ell} }
g_{\fz,r}] (\cdot, t-r) \big\|_{L^2(\bbR^d)}\,d\rho\,
\chi\ci{I_{\fz_{d+1},\ell}}(t-r) \,dt \,dr
\\
&\lc \sum_{\fz} \int\int_t \int_{\rho=1/2}^2|\omega_n(\rho)|d\rho \, |F(\fz,r)|  \times \\
 &\qquad\qquad 2^{\ell/2} r^{d-1}
\big\|[\chi\ci{R_{\dz,\ell} }
g_{\fz,r}] (\cdot, t-r) \big\|_{L^2(\bbR^d)}\,
\chi\ci{I_{\fz_{d+1},\ell}}(t-r) \,dt \,dr.
\end{align*}
Now  we use \eqref{omeganL1}, apply the Cauchy-Schwarz inequality in $t$
and then use the
normalizing assumption \eqref{Gzr}
to bound the last expression by
\begin{align*}
&2^{\ell/2} \sum_{\fz} \int |F(\fz,r)|
 r^{d-1}\, 2^{\ell/2} \, \times
\\
&\qquad
\Big(\int \big\|[\chi\ci{R_{\dz,\ell} }
g_{\fz,r}] (\cdot, t-r) \big\|_{L^2(\bbR^d)}^2
\chi\ci{I_{\fz_{d+1},\ell}}(t-r)\, dt \Big)^{1/2} dr
\\&\lc \,2^{\ell} \sum_{\fz} \int|F(\fz,r)|
r^{d-1} dr\,.
 \end{align*}
This gives the assertion for $p=\nu=1$, when $\alpha= \frac{d-1}2$.
\end{proof}

\section{The main estimate}\label{mainestimate}

We formulate our main estimate which will yield both  the implications
(iv)$\implies$(i) and (iv)$\implies$(ii) of Theorem \ref{dLconverse}.
In this section $\chi_1$  will be a $C^\infty$ function supported
in $(5/8, 17/8)$ and $\chi$ will be a $C^\infty$  function supported in
$(-4,4)$. We now consider the convolution operator $T$ on $\bbR^{d+1}$
with multiplier
\Be\label{multiplier}
m(\xi,\tau)=\sum_{k\in \bbZ} \chi_1(2^{-k}|\xi|)\chi(2^{-k}\tau)
\gga_k\big(\frac{|\xi|-b_k \tau}{2^k}\big).
\Ee

\begin{theorem}\label{maintheorem}
Suppose that $1<p_1<\frac{d}{d+1}$ and that Hypothesis
$\text{Sph}(p_1,d)$
holds.  Let $m$ be as in \eqref{multiplier},  with $b_k\in\bbR$ and $|b_k|\le 2$. Let $1<p<p_1$ and $p\le \nu\le\infty$ and assume that
\Be\label{Cpnu}
\fC_{p,\nu} :=
 \sup_k \biggl\|\frac{\widehat \gga_k}
 { (1+|\cdot|)^{\frac{d-1}2}}
\biggl\|_{L^{p,\nu}(\bbR,
   (1+|\cdot|)^{d-1}dr)}\, <\, \infty\,.
\Ee
Then
\Be \big\| \cF^{-1}[ m\widehat f]\big\|_{L^{p,\nu}(\bbR^{d+1})}\lc
  \fC_{p,\nu}
\|f\|_{L^{p}(\bbR^{d+1})}\,.
\Ee
\end{theorem}

We apply the Fourier inversion formula on the real line to $\gga_k$
and get
\begin{align*}
m(\xi,\tau)&=
\sum_k \chi(2^{-k}\tau)\chi_1(2^{-k}|\xi|)
\frac{1}{2\pi} \int \widehat \gga_k(s)
e^{i2^{-k}(|\xi|-b_k\tau)s}
 ds\,.
\end{align*}

By standard singular integral theory the convolution operator with Fourier
multiplier
$$\sum_k \chi(2^{-k}\tau)\chi_1(2^{-k}|\xi|)
\frac{1}{2\pi} \int_{-2}^2 \widehat \gga_k(s)
e^{i2^{-k}(|\xi|-b_k\tau)s}\,ds
$$
is bounded on $L^p(\bbR^{d+1})$ for all $p\in (1,\infty)$.
Therefore it suffices  to consider the Fourier multiplier
\Be \label{mult2}\sum_k \chi(2^{-k}\tau)\chi_1(2^{-k}|\xi|)
\int_2^\infty
\widehat \gga_k(s)
\exp(is2^{-k}(|\xi|-\tau)) ds
\Ee and a similar multiplier involving an integration over $(-\infty,-2)$.

We note    that these multipliers  define bounded
functions. Their  $L^\infty(\bbR^{d+1})$ norms are  bounded by
\Be \label{Bconst}  \sup_k\int|\widehat\gga_k(s)|\,ds \,
\lc \,\sup_k
\biggl\|\frac{\widehat\gga_k}{(1+|\cdot|)^{\frac{d-1}{2}}}
\biggr\|_{L^{p,\infty}(\nu_d)},
\qquad p<\frac{2d}{d+1};\Ee
here  $\nu_d$
denotes the measure $$d\nu_d(s)=(1+|s|)^{d-1}ds.$$  To see
\eqref{Bconst}
note that  the function $s\mapsto
(1+|s|)^{-\frac{d-1}2}$ belongs to $L^{q,1}(\nu_d)$ if
$q>\frac{2d}{d-1}$.
 Thus, for
$p<\frac{2d}{d+1}$,
we have
\[
\int|w(s)|ds=\int \frac{|w(s)|}{(1+|s|)^{\frac{d-1}{2}}}
\frac{1}{(1+|s|)^{\frac{d-1}{2}}}  d\nu_d(s)\,\lc\,
\biggl
\|\frac{w}{(1+|\cdot|)^{\frac{d-1}{2}}}\biggl\|_{L^{p,\infty}(\nu_d)},
\]
which implies \eqref{Bconst}.

Now let  $\vth$ be a $C^\infty$-function on the real line
 supported in
$(1/8,8)$ so that $\vth(s)=1$ on $(1/5,5)$ and observe that multiplication
with $\vth(\beta|\xi|) $ does not affect $\chi_1(|\xi|)$ as
long as $1/2\le \beta\le 2$.
Thus we have to prove that the convolution operator with multiplier
\Be\label{mult3}\sum_k \chi(\tfrac{\tau}{2^{k}})\chi_1(\tfrac{|\xi|}{2^k})\sum_{n=1}^\infty
\int_{2^n}^{2^{n+1}}\vth(2^{-n} s \tfrac{|\xi|}{2^k}))
\widehat \gga_k(s)
\exp(is\tfrac{|\xi|-b_k\tau}{2^k})\, ds
\Ee is bounded on $L^p(\bbR^{d+1})$.


One can express
$\cF^{-1}_d[ \exp(\pm i |\cdot|)  \vth(2^{-n}|\cdot|)]
$ as an integral over spherical means plus an error term:

\begin{lemma} \label{avelemma}
For $n\ge 1$,
\begin{equation}\label{splitting}
\cF^{-1}_d[ e^{\pm i |\cdot|}  \vth(2^{-n}|\cdot|)]= 2^{n(d-1)/2} \int_{1/2}^2
  \om_n^\pm(\rho) \sigma_\rho d\rho\,+\, \widetilde E_n^\pm
\end{equation}
where $\om_n^\pm$ is smooth on $(1/2,2)$
\begin{equation} \label{uniformL1}
\sup_n \int|\om_n^\pm(\rho)|d\rho<\infty,
\end{equation}
and, for any $N$,
\Be\label{Endecay}|\widetilde E_n^\pm(x)|+2^{-n}|\nabla \widetilde E_n^\pm(x)| \le C_N 2^{-nN}(1+|x|)^{-N}.
\Ee
\end{lemma}
This can be proven by an application of the stationary phase method; a
more direct argument is given in   Lemma  10.2 in \cite{hns}.

From the lemma we see that the convolution operator with multiplier
\eqref{mult3} can be split as
$$\sum_k  \cK_k*f +
\sum_{n=1}^\infty \sum_{k\in \bbZ}
2^{k(d+1)}E_{n,k}(2^k\cdot)*  f$$
where the main term is obtained by
substituting   the first term in
\eqref{splitting} for $\cF^{-1}_d[ e^{\pm i|\cdot|}
  \vth(2^{-n}|\cdot|)]$ 
(\cf. \eqref{Kkdefin} below)
and thus the  rescaled term $E_{n,k}$ is given by
\begin{multline}\label{kernelerror}
 E_{n,k}(x,t) = \\
 \int_{2^n}^{2^{n+1}}\widehat\gga_k(s)
\int \zeta(x-y,t-b_k s)
s^{-d} \widetilde E_n(s^{-1}y) \,dy\, ds
\end{multline}
where $\widehat \zeta(\xi,\tau)=
\chi(\tau) \chi_1(|\xi|) $.

From \eqref{Endecay} one gets
\begin{multline*}|E_{n,k}(x,t)|+|\nabla E_{n,k} (x,t)| \,\le
 C_{N,j} 2^{-nN} \,\times \\
 \int_{s=2^n}^{2^{n+1}} |\widehat\gga_k(s)|
\int (1+|x-y|+ |t-b_ks|)^{-d-2} (1+2^{-n}|y|)^{-N/2}dy\,
ds.
\end{multline*}
We use
$(1+|x-y|+ |t-b_ks|)^{-d-2} \lc
(\frac{1+|y|+|b_ks|}{1+|x|+|t|})^{d+2}$ and since  $|b_k|\le 2$
this implies (assuming $N$ is chosen sufficiently large, say $N>10 d$)
$$|E_{n,k}(x,t)|+
|\nabla E_{n,k} (x,t)| \lc \|\widehat \gga_k\|_{L^1(\bbR)} \,2^{-n} (1+|x|+|t|)^{-d-2}.$$
From this estimate it follows easily that
the operator
$\cE_n$ defined by
$$\cE_n f=
\sum_{k\in \bbZ}
2^{k(d+1)}E_{n,k}(2^k\cdot)*  f$$
is a Calder\'on-Zygmund operator which is bounded on $L^p(\bbR^{d+1})$ and the sum of the operator norms $\sum_{n=1}^\infty \|\cE_n\|_{L^p\to L^p}$ is
bounded by a constant only depending on $p$.

We now consider the main term. This is  the operator
of convolution on $\bbR^{d+1}$ with the kernel
$\sum_k \cK_k$ where
\begin{multline}\label{Kkdefin}
\cF_{d+1} [\cK_k] (\xi,\tau)=
\chi_1(2^{-k}|\xi|)\chi(2^{-k}\tau) \,\times \\
\sum_{n=1}^\infty 2^{n\frac{d-1}{2}}
\int_{2^n}^{2^{n+1}}\widehat \gga_k(s)e^{-ib_k s2^{-k}\tau}
\int_{1/2}^2 \om_n(\rho)
\cF_d[\sigma_\rho](2^{-k}s\xi)\,
d\rho\, ds\,.
\end{multline}

%

We now let $\psi_\circ$, $\psi$ be $C^\infty_0$-functions as defined in the
introduction
and define $\eta_\circ\in \cS(\bbR^{d+1})$ by
$$\widehat \eta_\circ(\xi,\tau)=
\frac{\chi_1(|\xi|)\chi(\tau)}{(\widehat \psi_\circ(\xi))^4}
=\frac{\chi_1(|\xi|)\chi(\tau)}{(\widehat \psi(\xi))^2}\,.
$$
Define the dyadic Littlewood-Paley operator
$L_k$ by $$\cF_{d+1}[L_k f](\xi,\tau)=
\widehat \eta_\circ(2^{-k}\xi,2^{-k}\tau)
\cF_{d+1} [f](\xi,\tau)\,
.$$
Then
$$\cK_k*f(x,t)= \int_2^\infty \int 2^{kd} H_{k,s}(2^k(x-y)) P_k L_kf
(y, t-2^{-k} b_k s)\,dy\, ds$$
where \begin{equation} \label{Pkdefinition} \cF_d[P_k g](\xi)= \psi(2^{-k}\xi) \cF_d g(\xi)
\end{equation}
 and
\begin{align} \notag
H_{k,s }(x)&= \sum_{n=1}^\infty 2^{n\frac{d-1}{2}}\widehat\gga_k(s)
\chi_{[2^n, 2^{n+1})}(s)
\int_{\rho=1/2}^2 \omega_n(\rho) [\psi* s^{-d}\sigma_\rho(s^{-1}\cdot)]d\rho
\\ \label{Hks}
&=\sum_{n=1}^\infty \widehat\gga_k(s)\chi_{[2^n, 2^{n+1})}(s) 2^{n\frac{d-1}{2}}
s^{1-d}
\int_{\rho=1/2}^2 \omega_n(\rho)
 [\psi* \sigma_{\rho s}]d\rho\,;
\end{align}
 in \eqref{Hks} the
$*$ is used  for convolution in $\bbR^d$.

\medskip

\noi {\bf Atomic decompositions.}
As in \cite{hns} we use atomic decompositions constructed from
a nontangential Peetre type  maximal square function
({\it cf.} \cite{peetre}, \cite{triebel}
and \cite{se-stud}),
$$
\fS f (x,t) = \Big(\sum_k \sup_{|(y,s)|\le 100 (d+1)\cdot 2^{-k}}
|L_k f(x+y, t+s)|^2\Big)^{1/2}.
$$
Then $\|\fS f\|_p\le C_p \|f\|_p$ for $1<p<\infty$.

For fixed $k$, we tile $\bbR^{d+1}$ by the dyadic cubes of sidelength
$2^{-k}$.
We  write $L(Q)=-k$
if we want to indicate that the sidelength of a dyadic cube is $2^{-k}$.
For each integer $j$, we introduce the set
$$\Omega_j=\{(x,t): \fS f(x,t)>2^j\}.$$ Let $\cQ^k_j$ be the set
of all dyadic cubes  of sidelength $2^{-k}$ which have the property that
$|Q\cap \Omega_j|\ge |Q|/2$ but
$|Q\cap \Omega_{j+1}|<|Q|/2$.
We also set
$$ \Omega_j^*= \{(x,t): M \chi\ci{\Omega_j}(x,t)>100^{-d-1}\}$$
where $M$ is the Hardy-Littlewood maximal operator. $\Omega_j^*$ is an open set
containing $\Omega_j$ and $|\Omega_j^*|\lc |\Omega_j|$.

Let $\cW_j$ is the set of all dyadic cubes $W$ for which the $20$-fold dilate of $W$ is contained in $\Omega_j^*$
and $W$ is maximal with respect to this property.
Clearly the interiors of these cubes are disjoint and we shall refer
to them as Whitney cubes for $\Omega_j^*$.
For such a Whitney cube $W\in\cW_j$ we denote by $W^*$ the tenfold
dilate of $W$, and observe that the family of dilates $\{W^*: W\in
\cW_j\}$ have bounded overlap.

Note that each $Q\in \cQ^k_j$ is contained in a unique $W\in \cW_j$.
For each
$W\in \cW_j$,  set
$$A_{k,W,j} = \sum_{\substack{Q\in\cQ^k_j\\Q\subset W}} L_kf \chi\ci Q;$$
note that only terms with $L(W)+k\ge 0$ occur.
Since any
any dyadic  cube $W$ can be a
Whitney cube for several $\Omega_j^*$ we also define  ``cumulative atoms'',
$$A_{k,W}=\sum_{j: W\in \cW_j} A_{k,W,j}.$$

Standard facts about these atoms are summarized in
\begin{lemma}
\label{atomsL2bd}
For each $j\in \bbZ$ the following inequalities hold.

(i)
$$
\sum_{W\in \cW_j} \sum_{k}\|A_{k,W,j}\|_2^2 \lc 2^{2j} \meas(\Omega_j).
$$

(ii) There is a constant $C_d$ such that for every  assignment
$W\mapsto k(W)\in\bbZ$,  defined for  $W\in \cW_j$,
and for
$0\le p\le 2$,
$$
\sum_{W\in \cW_j} \meas(W)\|A_{k(W),W,j}\|_\infty^p \le C_d 2^{pj} \meas(\Omega_j).
$$
\end{lemma}
For the proof see Lemma 7.1 in \cite{hns} (or related statements in
\cite{crf}, \cite{se-stud}).

With this notation it is now our task to show the inequality
\begin{equation}\label{all}
\Big\|\sum_k
\sum_j \sum_{\ell\ge 0}\sum_{\substack{W\in \cW_j\\
L(W)=-k+\ell}}
\cK_k
* A_{k,W,j}\Big\|_{L^{p,\nu} }  \lc \fC_{p,\nu} \|\fS f\|_p\,.
\end{equation}

Let
\Be\label{fV}
\fV_{k,s,W,j}(x,t)\,:=\,
\int 2^{kd} H_{k,s}(2^k(\cdot-y)) A_{k,W,j}
(y, t-2^{-k}b_k s)\,dy\,, 
\Ee
with $H_{k,s}$ in \eqref{Hks} and note that
\[
\cK_k
* A_{k,W,j} = P_k\Big[\int_2^\infty \fV_{k,s,W,j}\,ds\Big],
\]
with $P_k$ in \eqref{Pkdefinition}.

The estimate \eqref{all} follows then from
a  short range  and a long range inequality. The  short range
inequality
is
\begin{multline}
\label{shortr}
\Big\|\sum_k
\sum_j \sum_{\ell\ge 0}\sum_{\substack{W\in \cW_j\\
L(W)=-k+\ell}}
P_k\big [
\int_2^{2^{\ell}} \fV_{k,s,W,j}\,
ds \,\big]
\Big\|_{L^{p}} \\ \lc \,\sup_k \big\|\widehat \gga_k\big\|_{L^1(\bbR)}
\,\|\fS f\|_p, \qquad 1<p<2,
\end{multline}
and  implies the analogous $L^p\to L^{p,\nu}$ estimate since by
assumption $\nu\ge p$. Recall  that
$\sup_k \big\|\widehat \gga_k\big\|_{L^1(\bbR)}\lc
   \fC_{p,\infty}\lc
 \fC_{p,\nu}$ for
$p<\frac{2d}{d+1}$, {\it cf.} \eqref{Bconst}.

The  long range inequality is
\begin{multline}
\label{longr}
\Big\|\sum_k
\sum_j \sum_{\ell\ge 0}\sum_{\substack{W\in \cW_j\\
L(W)=-k+\ell}}
P_k\big [
\int_{2^{\ell}}^\infty
\fV_{k,s,W,j}\,
ds \,\big]
\Big\|_{L^{p,\nu} }  \lc \fC_{p,\nu} \|Sf\|_p\,.
\end{multline}

\medskip

\noi{\bf The short range estimate.} Since $\sum_j 2^{jp}
\meas(\Om_j)\lc\|\fS f\|_p^p$ it suffices
by Lemma \ref{dyadicinterpol} to show that for fixed $j$
and  $1<p<2$
\begin{multline}
\label{shortrfixedj}
\Big\|\sum_k
 \sum_{\ell\ge 0}\sum_{\substack{W\in \cW_j\\
L(W)=-k+\ell}}
P_k\big [
\int_2^{2^{\ell}}  \fV_{k,s,W,j}\,
ds \big]
\Big\|_{L^{p}}^p \\ \lc \,\sup_k \big\|\widehat \gga_k\big\|_{L^1(\bbR)}^p
\,2^{jp} \meas (\Omega_j).
\end{multline}
Here we estimate an expression which is supported  in
$\Omega_j^*$.
Thus the left hand side of \eqref{shortrfixedj} is dominated by
\begin{multline}
\label{shortrfixedjL2}
\meas (\Omega_j^*)^{1-p/2}\Big\|\sum_k
 \sum_{\ell\ge 0}\sum_{\substack{W\in \cW_j\\
L(W)=-k+\ell}}
P_k\big [
\int_2^{2^{\ell}} \fV_{k,s,W,j}\,
ds\,\big]
\Big\|_{L^{2}}^p
\end{multline}
which by the almost orthogonality of the operators $P_k$ is dominated
by a constant times
\begin{multline}
\label{shortrfixedjL2orth}
\meas (\Omega_j^*)^{1-p/2}\Big(\sum_k\Big\|
 \sum_{\ell\ge 0}\sum_{\substack{W\in \cW_j\\
L(W)=-k+\ell}}
\int_2^{2^{\ell}} \fV_{k,s,W,j}
\, ds
\Big\|_{L^{2}}^2\Big)^{p/2}\,.
\end{multline}
Now,  for fixed $W$ with $L(W)=-k+\ell$,  and  for every $s\le 2^\ell$,
the expression
$\fV_{k,s,W,j}$ is supported in the expanded cube $W^*$.
The
cubes $W^*$ with $W\in \Omega_j$  have bounded overlap, and therefore
the expression
\eqref{shortrfixedjL2orth} is dominated
by a constant times
\begin{multline}
\label{shortrsumW}
\meas (\Omega_j^*)^{1-p/2}
\Big(\sum_k\sum_{\ell\ge 0}\sum_{\substack{W\in \cW_j\\
L(W)=-k+\ell}}
\Big\|
 \int_2^{2^{\ell}} \fV_{k,s,W,j}
\, ds
\Big\|_{L^{2}}^2\Big)^{p/2}\,.
\end{multline}
Now we have for fixed $W$
\begin{align*}
&\Big\|
 \int_2^{2^{\ell}} \fV_{k,s,W,j}\, ds
\Big\|_{L^{2}(\bbR^{d+1})}\\
&\quad \lc \int_2^{2^\ell} \Big(\int
\big\|
2^{kd} H_{k,s}(2^k\cdot) * A_{k,W,j}(\cdot, t-2^{-k} b_k s)
 \big\|_{L^2(\bbR^d)}^2 dt\Big)^{1/2}\, ds
\\
&\quad \lc \int_2^{2^\ell} \Big(\int
\big\|
\cF_d[ H_{k,s}]\big\|_\infty^2 \big\| A_{k,W,j}(\cdot, t-2^{-k} b_k s)
 \big\|_{L^2(\bbR^d)}^2 dt\Big)^{1/2} ds
\\
&\quad=  \int_2^{2^{\ell}} \big \|\cF_d[ H_{k,s}]\big\|_\infty\, ds \, \|A_{k,W,j}\|_{L^2(\bbR^{d+1})}
\end{align*}
 and
\begin{multline*}
\int_2^{2^{\ell}} \big \|\cF_d [H_{k,s}]\big\|_\infty  ds
\lc\\ \sum_{n=1}^{\ell-1}\int_{2^n}^{2^{n+1}}
| \widehat\gga_k(s)| 2^{n\frac{d-1}{2}}
s^{1-d}\int_{\rho=1/2}^2 \omega_n(\rho) \big\|\F_d[\psi* \widehat
\sigma_{\rho s}]
\big\|_\infty \,d\rho\, ds\,.
\end{multline*}
Since
$\F_d[\psi* \widehat
\sigma_{\rho s}](\xi)= O(2^{n(d-1)/2})$ uniformly in $\rho\in (1/2,2)$
and $s\in[2^n,2^{n+1}]$. Since $\sup_n \|\om_n\|_1\le 1$  we get
\[ \int_2^{2^{\ell}}  \big\|\cF_d [H_{k,s}]\big\|_\infty ds \lc \int|
\widehat\gga_k(s)| ds\,
\]
and thus
\Be \label{fVest}
\Big\|
 \int_2^{2^{\ell}} \fV_{k,s,W,j}
\, ds
\Big\|_{L^{2}} \lc \int|
\widehat\gga_k(s)| ds\,
\| A_{k,W,j}
\|_2\,.
\Ee
We use this estimate in  \eqref{shortrsumW}.
By  Lemma \eqref{atomsL2bd} we have
\[\sum_k\sum_{\ell\ge 0}\sum_{\substack{W\in \cW_j\\
L(W)=-k+\ell}}
\big\| A_{k,W,j}\big\|_2^2 \lc
\sum_{W\in \cW_j}\sum_k\| A_{k,W,j}\big\|_2^2 \lc
2^{2j}\meas(\Omega_j).\]
We combine this with \eqref{fVest}. Since $\meas(\Omega_j^*)\lc
\meas(\Omega_j)$ it follows that the right hand side of
\eqref{shortrsumW}
is dominated by a constant times
\[ \Big[\sup_k \int|\widehat\gga_k(s)|\,ds\Big]^p \,  \meas(\Omega_j) 2^{jp}\]
which then yields \eqref{shortrfixedj} and finishes the proof of the
short range estimate.

\medskip

\noi{\bf The long range estimate.}
It is now advantageous to use the cumulative atoms $A_{k,W}$.
If we let
\Be\label{fVsumj}
\fV_{k,s,W}(x,t)\,:=\,
\int 2^{kd} H_{k,s}(2^k(\cdot-y)) A_{k,W}
(y, t-2^{-k} b_ks)\,dy\,
\Ee
then $\fV_{k,s,W}\,=\sum_j \fV_{k,s,W,j}$ and we have to show
\begin{multline}
\label{longrsumj}
\Big\|\sum_k
\sum_{\ell\ge 0}\sum_{\substack{W:\\
L(W)=-k+\ell}}
P_k\big [
\int_{2^{\ell}}^\infty
\fV_{k,s,W}
ds \,\big]
\Big\|_{L^{p,\nu} }^p  \lc \fC_{p,\nu}^p\sum_j\meas(\Omega_j) 2^{jp}\,.
\end{multline}
By Minkowski's inequality  this follows from  estimates for fixed
$\ell>0$, with exponential decay:
\begin{multline}
\label{longrsumjell}
\Big\|\sum_k
\sum_{\substack{W:\\
L(W)=-k+\ell}}
P_k\big [
\int_{2^{\ell}}^\infty
\fV_{k,s,W}
ds \,\big]
\Big\|_{L^{p,\nu} }  \\
\lc \fC_{p,\nu} 2^{-\ell\alpha(p)}
 \Big(\sum_j\meas(\Omega_j) 2^{jp}\Big)^{1/p}\,.
\end{multline}
Here $\alpha(p)>0$ for $p<p_1$ (in fact $\alpha$ will be as in Proposition \ref{propext}).

We interpolate an $L^1(\ell^1)\to L^1$ inequality and an
$L^2(\ell^2)\to L^2$ inequality for the operators $P_k$.
Let $\fm$ denote the measure on $\Bbb R^{d+1}\times \bbZ$ defined as
 the product measure of Lebesgue measure on $\bbR^{d+1}$ and
counting measure on $\bbZ$. Define for (suitable) functions $h$ on
$\bbR^{d+1}\times \bbZ$ an operator $P$ by  $P h(x,t)=\sum_k \F^{-1}_d
[\widehat \psi(2^{-k}\cdot) \F_d h(\cdot,k)](x,t)
$. Then $P$ maps
$L^1(\Bbb R^{d+1}\times \bbZ,\fm)$ to $L^1(\bbR^{d+1})$
 and by orthogonality
$L^2(\Bbb R^{d+1}\times \bbZ,\fm)$ to $L^2(\bbR^{d+1})$; thus for
 $1<p<2$
\[
\|Ph\|_{L^{p,\nu} (\bbR^{d+1})} 
\lc
\big\|h\big\|_{L^{p,\nu}
(\Bbb R^{d+1}\times \bbZ,\fm)}\,.
\]
Now by Lemma \ref{fubinilemma}
this also implies, under the additional
restriction $\nu\ge p$,
\[\Big\|\sum_k P_kf_k\Big\|_{L^{p,\nu} (\bbR^{d+1})} \lc
\Big(\sum_k
\| f_k\|_{L^{p,\nu} (\bbR^{d+1})}^p\Big)^{1/p}.
\]
 Using this inequality we see that
\eqref{longrsumjell} follows from
\begin{multline}
\label{longrsumjelllp}
\sum_k\Big\|
\sum_{\substack{W:\\
L(W)=-k+\ell}}
\int_{2^{\ell}}^\infty
\fV_{k,s,W}
ds \,
\Big\|_{L^{p,\nu} }^p  \lc \fC_{p,\nu}^p 2^{-\ell p\alpha(p)}
 \sum_j\meas(\Omega_j) 2^{jp}\,.
\end{multline}

We need to  rewrite $\int_{2^{\ell}}^\infty
\fV_{k,s,W}\,ds \,$ and also scale it
in order to apply Hypothesis $\text{Sph}(p_1,d)$ (or rather
 its consequence stated as Proposition \ref{propext}).
Note that
\[\fV_{k,s,W}(x,t)=\int_{2^\ell}^\infty \int H_{k,s}(2^k x-y) A_{k,W}
(2^{-k}y, 2^{-k}(2^k t-b_ks)) \,dy\,ds.\]
If we set
$$
a_{k,W}(y,u)=A_{k,W}(2^{-k}y, 2^{-k}u)
$$
and
$$
\fv_{k,s,W}(x,t)=
\int H_{k,s}( x-y) a_{k,W}
(y, t-b_ks) \,dy\,ds
$$
then
$ \fV_{k,s,W}(x,t)= \fv_{k,s,W}(2^kx, 2^kt)$
and of course we have
\Be
\label{LpnunormofintfV}
\biggl\|
\sum_{\substack{W:\\
L(W)=-k+\ell}}
\int_{2^{\ell}}^\infty
\fV_{k,s,W}
ds \,
\biggr\|_{L^{p,\nu} }= 2^{-k(d+1)/p}
\biggl\|
\sum_{\substack{W:\\
L(W)=-k+\ell}}
\int_{2^{\ell}}^\infty
\fv_{k,s,W}
ds \,
\biggr\|_{L^{p,\nu} }^p\,.
\Ee
Next (with $*$ denoting convolution in $\bbR^d$)
\begin{multline} \label{smallfv}\int_{2^\ell}^\infty
\fv_{\ka, r,W}(x,t) dr =\\
\int_{1/2}^2 \sum_{n=\ell}^\infty \om_n(\rho)
\int_{2^n}^{2^{n+1}}
\widehat \gga_k(r) r^{1-d} 2^{n(d-1)/2}
[\psi* \sigma_{\rho  r} *a_{k,W}](x, t-b_kr)\, dr\,d\rho\,.
\end{multline}

We are now in the position to apply Proposition \ref{propext},
with the choice of
\begin{multline*}
F(\fz,r)\equiv  F_{k,\ell}(\fz,r)\,=\, \\ \sum_{n=\ell}^\infty\widehat \gga_k(r) r^{1-d} 2^{n\frac{d-1}2}
\chi_{[2^n, 2^{n+1}]}(r)
\Big\|\sum_{W: L(W)=\ell-k} A_{k,W}(2^{-k }\cdot)\chi_{z,\ell}
\Big\|_2\,.
\end{multline*}

The sum in $W$  collapses as for given $\fz=(\dz,z_{d+1})$ there is a unique dyadic
 cube $W$ of sidelength $2^{\ell-k}$ so that the 
 dyadic cube $2^k W=\{2^k (x,t): (x,t)\in W\}$ is equal to
$R_{\dz,\ell}\times I_{\fz_{d+1},\ell}$.
Also observe the pointwise estimate
$$
 \sum_{n=\ell}|\widehat \gga_k(r)|
r^{1-d} 2^{n\frac{d-1}2}
\chi_{[2^n, 2^{n+1}]}(r) \lc \frac{|\widehat \gga_k(r)|}{(1+|r|)^{\frac{d-1}{2}}}.
$$

We now proceed to finish the proof of \eqref{longrsumjelllp}.
By Proposition \ref{propext} and the Fubini-type Lemma \ref{fubinilemma}
we get from \eqref{smallfv}
\begin{align}\notag
&\Big\|
\sum_{\substack{W:\\
L(W)=-k+\ell}}
\int_{2^{\ell}}^\infty
\fv_{k,s,W}
ds \,
\Big\|_{L^{p,\nu} } \lc
\,2^{\ell ((d+1)(\frac 1p-\frac 12)-\alpha)}
\|F_{k,\ell} \|_{L^{p,\nu}(\bbZ^{d+1}\times\bbR^+\!,\,\mu_d)}
\\
\label{summinginW}
&\quad\lc 2^{\ell ((d+1)(\frac 1p-\frac 12)-\alpha)}
\biggl\|
\frac{|\widehat \gga_k|}{(1+|\cdot|)^{\frac{d-1}{2}}}
\biggr\|_{L^{p,\nu}(\bbR, (1+|r|)^{d-1}dr)}\quad \times
\\\notag  &\qquad\qquad
\Big(\sum_{\substack{W:\\
L(W)=-k+\ell}}\big\|A_{k,W}(2^{-k}\cdot)\big\|_2^p\Big)^{\frac 1p}
\end{align}
where $\alpha$ is as in Proposition \ref{propext}. 
Combining \eqref{LpnunormofintfV}
and \eqref{summinginW} we obtain after a change of variables
\begin{multline} \label{fixedkestimate}
\Big\|
\sum_{\substack{W:\\
L(W)=-k+\ell}}
\int_{2^{\ell}}^\infty
\fV_{k,s,W}
ds \,
\Big\|_{L^{p,\nu} }\\
\lc \, \fC_{p,\nu} 2^{-\ell\alpha}
2^{(\ell-k) (d+1)(\frac 1p-\frac 12)}
\Big(\sum_{\substack{W:\\
L(W)=-k+\ell}}\big\|A_{k,W}\big\|_2^p\Big)^{1/p}\,.
\end{multline}

Note that for fixed $k$ and $W$,
 the functions $ A_{k,W,j}$ live on disjoint sets (since the dyadic
 cubes of sidelength  $2^{-k}$ are disjoint and each such cube
 is in exactly one
 family $\cQ_j^k$).
Thus
$\|A_{k,W}\|^p_2 \lc\sum_j\|A_{k,W,j}\|^p_2.$
We  now sum in $k$ and obtain from
\eqref{fixedkestimate}
\begin{multline*}
\Big(\sum_k\Big\|
\sum_{\substack{W:\\
L(W)=-k+\ell}}
\int_{2^{\ell}}^\infty
\fV_{k,s,W}
ds \,
\Big\|_{L^{p,\nu} }^p\Big)^{1/p}
\\
\lc \, \fC_{p,\nu} 2^{-\ell\alpha}
\Big(\sum_k\sum_j\sum_{\substack{W\in \cW_j:\\
L(W)=-k+\ell}}\meas(W)^{1-p/2}
\big\|A_{k,W,j}\big\|_2^p\Big)^{1/p}\,.
\end{multline*}
Finally, using
 part (ii) of
Lemma \ref{atomsL2bd},  we get
\begin{align*}
&
\Big(\sum_k\sum_j\sum_{\substack{W\in \cW_j:\\
L(W)=-k+\ell}}\meas(W)^{1-p/2}\big\|A_{k,W,j}\big\|_2^p\Big)^{1/p}\,
\\
&\lc
\Big(\sum_j\sum_{W\in \cW_j}\meas(W)\big\|A_{\ell-L(W),W,j}
\big\|_\infty^p\Big)^{1/p}\,
\\
&\lc
\Big(\sum_j \meas(\Omega_j)\,2^{jp}\Big)^{1/p}\, \lc \,
\|\fS f\|_p\,.
\end{align*}
This finishes the proof of \eqref{longrsumjelllp}.

\qed

\section{Proof of Theorem \ref{dLconverse}}\label{pfdLconverse}
By the remarks in the introduction (following the statement of Corollary
\ref{dLconversecor}) it only remains to be shown that (iv) implies (i)
and (ii).
These implications quickly follow from Theorem \ref{maintheorem}.

For the  implication (iv)$\implies$(i) 
we show, for the choices $b=1$ and $b=\sqrt 2$  that the multiplier 
\Be\label{splittingintervals}
m(\xi,\tau)= \sum_k \indicator{[2^k, 2^{k}\sqrt 2)} (\tau)
\gamma_k\big(2^{-k}(|\xi|-b\tau))\Ee
defines an operator which is bounded from $L^p(\bbR^{d+1})$ to
$L^{p,\nu}(\bbR^{d+1})$. The choice $b=\sqrt 2$ and scaling in $\tau$
then also covers the multiplier 
$$m(\xi,\tau)= \sum_k \indicator{[2^{k}\sqrt 2, 2^{k+1})}(\tau) 
\gamma_k\big(2^{-k}(|\xi|-\tau))$$ and the assertion follows.

For the proof of  \eqref{splittingintervals}
pick a smooth function $\chi_2$ which is equal to one on $[1,\sqrt 2]$
and supported in $(9/10, 3/2)$. Recall that $\ga_k$ is supported in
$(-1/4,1/4)$and pick a smooth function $\chi_1$ which is equal to one
on $(\frac{9}{10} b-\frac 14, \frac 32 b+\frac 14)$ and supported on 
$(\frac{b}{2}, 2b)$. Observe that, with these definitions
\Be\label{factoring}
m(\xi,\tau)= \chi_E(\tau)
\sum_k \chi_2(2^{-k}\tau) \chi_1(2^{-k}|\xi|)
\gamma_k\big(2^{-k}(|\xi|-b\tau))
\Ee 
where $E=\bigcup_{k\in \bbZ} [2^k, 2^{k+\frac 12})$. By the Marcinkiewicz
  multiplier theorem the convolution with multiplier $\chi_E(\tau)$ is bounded
on $L^{p,\nu}(\bbR^{d+1})$ for all $1<p<\infty$, $0<\nu\le \infty$.
Therefore it suffices to prove that under condition
\eqref{gammakhat} the multiplier
\Be \label{mdef} m(\xi,\tau) = \sum_k \gamma_k\big(2^{-k}(|\xi|-\tau))
\chi_2(2^{-k}\tau)\chi_1(2^{-k}|\xi|)
\Ee
defines a convolution which maps $L^p(\Bbb R^{d+1})$ to
$L^{p,\nu}(\bbR^{d+1})$.
But this follows immediately from Theorem \ref{maintheorem}, with the choice
of $\Gamma_k=\gamma_k$, and $b_k=b$ ($=1$ or $\sqrt 2$)  for all $k\in \bbZ$.

Next, for the implication (iv)$\implies$(ii) we first note that since
$\tau_k\in [2^k, 2^{k+1}]$ the term $\gamma(2^{-k}(|\xi|-\tau_k))$
vanishes for $|\xi|\notin (\frac 34 2^k, \frac 94 2^k)$. 
Now choose $\chi_1$ so that $\chi_1$ is supported in $(1/2, 5/2)$ and
equal to one on $(3/4, 9/4)$. 
Then $$\cF_d[\sum_k\alpha_k T^{\tau_k} \!f](\xi)=
\sum_k \alpha_k \gamma_k(2^{-k}|\xi|- 2^{-k}\tau_k)\chi_1(2^{-k}|\xi|) \F_d[f](\xi)\,.$$
Now let  $\chi $ be smooth and compactly
supported in $(-4,4)$. We claim that the  multiplier transformation with Fourier muliplier
\Be \label{secondappl}
M(\xi,\tau)=\sum_k \gamma_k\big(2^{-k}(|\xi|-\tau_k))
\chi_1(2^{-k}|\xi|)\,\chi(2^{-k}\tau)
\Ee
maps $L^p(\bbR^{d+1})$ to $L^{p,\nu}(\bbR^{d+1})$.  To see this we apply
 Theorem \ref{maintheorem} with  
$\Gamma_k(s)=\alpha_k \gamma_k (s-  2^{-k}\tau_k)$ and $b_k=0$ for all
$k\in \bbZ$. The condition  \eqref{gammakhat} for
$\widehat{\gamma_k}$ is obviously equivalent with the condition
 \eqref{Cpnu}
for $\widehat \gga_k$.

Now in \eqref{secondappl}  $\chi(\tau)$ may be chosen 
so that  $\chi(0)=1 $.
With this choice it follows  by de Leeuw's theorem (Lemma
\ref{leeuwlemma}) that $\sum_k \alpha_k T^{\tau_k}$ maps $L^p(\bbR^d)$
to $L^{p,\nu}(\bbR^d)$. \qed

\section{The cone multiplier}\label{conemultwt}

\begin{proof}[Proof of Corollary \ref{wtbr}]
It suffices to consider the multiplier $\rho_\lambda(\xi,\tau) \chi_{(0,\infty)}(\tau)$.
We split for $\tau>0$
$$\rho_\lambda(\xi,\tau)= \sum_{k\in \bbZ} \indicator{[2^k,
  2^{k+1})}(\tau)
\frac{2^{k\la}(\tau+|\xi|)^\la}{\tau^{2\la}}
\Big(\frac{\tau-|\xi|}{2^k}\Big)_+^\la.$$
Now let $b\in C^\infty_c(\bbR)$ be supported in $(-1/4,4)$ and satisfy $b(s)=1$
for $|s|\le 1/8$.
We can then write
$$
\rho_\lambda(\xi,\tau)\chi_{(0,\infty)}(\tau)=
a_\lambda(\xi,\tau)
\sum_{k\in \bbZ}\indicator{[2^k,
  2^{k+1})}(\tau)  
\gamma\big(\frac{|\xi|-\tau}{2^k}\big)
+ \widetilde a_\lambda(\xi,\tau)
$$
where
\begin{align*}
a_\lambda(\xi,\tau)&= \sum_{k\in \bbZ}\indicator{[2^k,
  2^{k+1})}(\tau) 
\Big( \frac{2^k(\tau+|\xi|)}{\tau^2}\Big)^\la
b\big(\frac{|\xi|-\tau}{2^k}\big) \\
\widetilde a_\la(\xi,\tau)&=
\sum_{k\in \bbZ}\indicator{[2^k,
  2^{k+1})}(\tau) 
\Big(1-\frac{|\xi|^2}{\tau^2}\Big)_+^\la
\big(1- b\big(\frac{|\xi|-\tau}{2^k}\big)\big) 
 \end{align*}
and
$$\gamma(u)= \begin{cases} (-u)^\la b(u) &\text{ for } u<0
\\
0 &\text{ for } u>0
\end{cases}\,.
$$

The multipliers $a_\la$ and $\widetilde a_\la$ are treated by the Marcinkiewicz multiplier theorem. The associated convolution operators are thus bounded on $L^{p,\nu}$ for all $1<p<\infty$, $0<\nu\le\infty$.
Therefore the corollary follows if we can show that the convolution operator with
multiplier
$$ \sum_{k\in \bbZ} \indicator{[2^k,
  2^{k+1})}(\tau) \gamma(2^{-k}(|\xi|-\tau))
$$
maps $L^p(\bbR^{d+1})$ boundedly to
$L^{p,\infty}(\bbR^{d+1})$. By  Theorem \ref{dLconverse} this is the case if
  \[
\biggl\|\frac{\widehat \gamma}
 { (1+|\cdot|)^{\frac{d-1}2}}
\biggl\|_{L^{p,\infty}(\bbR,
   (1+|\cdot|)^{d-1}dr)}<\infty.
\]
But
$$|\widehat \gamma(s)| \le C(1+|s|)^{-\la-1}
$$
and it is easy to check
that
$(1+|\cdot|)^{-\la-1-\frac{d-1}{2}}$ belongs to $L^{p,\infty}(\bbR,
   (1+|r|)^{d-1}dr)$
if and only if $\la\ge d/p-(d+1)/2$.
\end{proof}

\end{document}